\documentclass{amsart}
\usepackage{amsmath}
\usepackage{amssymb}

\newtheorem{theorem}{Theorem}[section]
\newtheorem{lemma}[theorem]{Lemma}
\theoremstyle{definition}
\newtheorem{definition}[theorem]{Definition}
\newtheorem{example}[theorem]{Example}

\newtheorem{proposition}[theorem]{Proposition}
\newtheorem{corollary}[theorem]{Corollary}
\theoremstyle{remark}
\newtheorem{remark}[theorem]{Remark}

\input{tcilatex}

\begin{document}
\title{\textbf{ON FUZZY SEMIHYPERRINGS}}
\author{\textbf{Aqeel Ahmed, Muhammad Aslam}}
\address{Department of Mathematics Quaid-i-Azam University Islamabad 45320
Paksitan.}
\email{draslamqau@yahoo.com.}
\address{Department of Mathematics Quaid-i-Azam University Islamabad 45320
Paksitan}
\email{aqeel\_ahmed709@yahoo.com.}
\address{Department of Mathematics \& Computer Science, Faculty of Natural
Sciences\\
University of Gjirokastra\\
Albania}
\email{kostaq\_hila@yahoo.com, khila@uogj.edu.al}

\begin{abstract}
In this article we introduce the study of fuzzy semihyperrings and fuzzy $R$%
-semihypermodules, where $R$ is a semihyperrings and $R$-semihypermodules
are represntations of $R$. In particular, semihyperrings all of whose
hyperideals are idempotent, called fully idempotent semihyperrings, are
investigated in a fuzzy context. It is proved, among other results, that a
semihyperring $R$ is fully idempotent if and only if the lattics of fuzzy
hyperideals of $R$ is distributive under the sum and product of fuzzy
hyperideals. It is also shown that the set of proper fuzzy prime hyperideals
of a fully idempotent semihyperring $R$ admits the structure of a
topological space, called the fuzzy prime spectrum of $R$.
\end{abstract}

\maketitle

\section{\protect\huge Introduction}

The concept of hyperstructure was first introduced by Marty \cite{17} at the
eighth Congress of Scandinavian Mathematicians in 1934, when he defined
hypergroups and started to analyze their properties. Now, the theory of
algebraic hyperstructures has become a well- established branch in algebraic
theory and it has extensive applications in many branches of mathematics and
applied science. Later on, people have developed the semi-hypergroups, which
are the simplest algebraic hyperstructures having closure and associative
properties. A comprehensive review of the theory of hyperstructures can be
found in \cite{8,6,24}. The canonical hypergroups are a special type of
hypergroup which Mittas \cite{19}, is the first one that studied them
extensively. The theory of hypermodules which their additive structure is
just a canonical hypergroup, several authors have studied that, for example,
Massouros \cite{18}, Corsini \cite{7}, Davvaz \cite{10,9}, Zhan et al. \cite%
{28}, Ameri \cite{2} and Zahedi and Ameri \cite{27}\pagebreak

The theory of fuzzy sets, introduced by Zadeh \cite{26} in 1965, has
provided a useful mathematical tool for describing the behavior of systems
that are too complex or ill-defined to admit precise mathematical analysis
by classical methods and tools. In\cite{3,3a,4,11,12,22,30,31}, some
applications of this theory in algebraic structures and hyperstructures can
be seen.

In this paper, however, we pursue an algebraic approach to investigate the
concept of fuzzy semihyperring and related notion in order to set the ground
for future work. In section 1, we provide basic definitions and establish
some preliminary results. In Section 2, we investigate fully idempotent
semihyperrings, that is, semihyperrings all of whose hyperideals are
idempotent. It is proved that such semihyperrings are characterized by the
property that each proper fuzzy hyperideal is the intersection of fuzzy
prime hyperideals containing it. In section 3. we construct the fuzzy prime
spectrum of fully idempotent semihyperrings in a manner analogous to the
construction of the prime spectrum in classical semiring theory.

\section{\protect\huge Preliminaries}

A hyperstructure is a non-empty set say $H$ together with a mapping "$\circ $%
": $H\times H\rightarrow P^{\ast }(H),$ where $P^{\ast }(H)$ is the set of
all the non-empty subsets of $H$ and "$\circ $" is called hyperoperation. If 
$x\in H$ and $A,B\in P^{\ast }(H),$ then by $A\circ B,$ $A\circ x,$ and $%
x\circ B,$ we mean $A\circ B=\cup \{a\circ b:a\in A$ and $b\in B\},$ $A\circ
x=A\circ \{x\}$ and $x\circ B=\{x\}\circ B,$ respectively \cite{8,6}. A
hypergroupoid is a set $H$ with a binary hyperoperation $\circ $ and a
commutative hypergroupoid $(H,\circ ),$ which is associative, that is $%
x\circ (y\circ z)=(x\circ y)\circ z,$ for all $x,y,z\in H$ is called a
semihypergroup. A hypergroup is a semihypergroup, such that, for all $x\in
H, $ we have $x\circ H=H=H\circ x,$ which is called reproduction axiom. If $%
H $ is a hypergroup and $K$ is a nonempty subset of $H.$ Then $K$ is a
subhypergroup of $H$ if $K$ itself a hypergroup under hyperoperation,
defined in $H.$ Hence it is clear that a subset $K$ of $H$ is a
subhypergroup if and only if $aK=Ka=K,$ under the hyperoperation on $H.$

A set $H$ togather a hyperoperation $\circ $ is called a polygroup, if the
following conditions are satisfied:

$(1)$ $x\circ (y\circ z)=(x\circ y)\circ z,$ for all $x,y,z\in H.$

$(2)$ There exist a unique element, $e\in H$ such that $e\circ x=x\circ e=x,$
for all $x\in H.$

$(3)$ For all $x\in H,$ there exist a unique element, say $x^{\prime }\in H$
such that $e\in x\circ x^{\prime }\cap x^{\prime }\circ x$ (where $x^{\prime
}=x^{-1}$)$.$

$(4)$ For all $x,y,z\in H,$ $z\in x\circ y\Rightarrow x\in z\circ
y^{-1}\Rightarrow y\in x^{-1}\circ z.$

A non-empty subset $K$ of a polygroup $(H,\circ )$ is called a subpolygroup
if $(K,\circ )$ is itself a polygroup. In this case we write $K<_{p}H.$

A commutative polygroup is called canonical hypergroup.

\begin{definition}
A semihyperring is an algebraic hypersystem ($R$,$\oplus $, $\cdot $)
consisting of a non-empty set $R$ together with one hyperoperation "$\oplus $%
" and one binary operation "$\cdot "$ on $R$, such that ($R$,$\oplus $) is a
commutative semihypergroup and ($R$,.) is a semigroup. For all $x,y,z\in R$,
the binary operation of multiplication is distributive over hyperoperation
from both sides that is, $x.(y\oplus z)=x.y\oplus x.z$ and $(x\oplus
y).z=x.z\oplus y.z$. An element $0\in R,$ is an absorbing element, such that 
$0\oplus x=x=x\oplus 0,$ and $0.x=0=x.0$ for all $x\in R$ \cite{13}$.$ By a
subsemihyperring of\ $R$, we mean a non-empty subset $S$ of $R$ such that
for all $x,y\in S,$ we have $x.y\in S$ and $x\oplus y\subseteq S.$
\end{definition}

A semihyperring ($R$,$\oplus $,$\cdot $) is called a hyperring if ($R$,$%
\oplus $) is a canonical hypergroup and ($R$,.) is a semigroup.

\begin{example}
Let $X$ be a non-empty finite set and $\tau $ is a topology on $X$. We
define the hyperoperation of the addition and the multiplication on $\tau $
as;
\end{example}

For any $A,B\in \tau ,$ $A\oplus B=A\cup B$ and $A.B=A\cap B.$ Then $(\tau
,\oplus ,.)$ is a semihyperring with absorbing element and additive identity 
$\Phi $ and multiplicative identity $X.$

\begin{example}
Let us consider a set $S=\left\{ \left[ 
\begin{array}{cc}
a & b \\ 
c & d%
\end{array}%
\right] ;a,b,c,d\in W\right\} $. Where $W$ is a set of whole numbers. We
define the hyperoperation of addition and multiplication as;

For $A=\left[ 
\begin{array}{cc}
a & b \\ 
c & d%
\end{array}%
\right] $ and $B=\left[ 
\begin{array}{cc}
a^{^{\prime }} & b^{^{\prime }} \\ 
c^{^{\prime }} & d^{^{\prime }}%
\end{array}%
\right] $ be taken from $S.$

$A\oplus B=\left[ 
\begin{array}{cc}
a & b \\ 
c & d%
\end{array}%
\right] \oplus \left[ 
\begin{array}{cc}
a^{^{\prime }} & b^{^{\prime }} \\ 
c^{^{\prime }} & d^{^{\prime }}%
\end{array}%
\right] =\left\{ \left[ 
\begin{array}{cc}
a+a^{^{\prime }} & b+b^{^{\prime }} \\ 
c+c^{^{\prime }} & d+d^{^{\prime }}%
\end{array}%
\right] \right\} \subseteq S$ and

$A.B=\left[ 
\begin{array}{cc}
a & b \\ 
c & d%
\end{array}%
\right] .\left[ 
\begin{array}{cc}
a^{^{\prime }} & b^{^{\prime }} \\ 
c^{^{\prime }} & d^{^{\prime }}%
\end{array}%
\right] =\left[ 
\begin{array}{cc}
aa^{^{\prime }}+bc^{^{\prime }} & ab^{^{\prime }}+bd^{^{\prime }} \\ 
ca^{^{\prime }}+dc^{^{\prime }} & cb^{^{\prime }}+dd^{^{\prime }}%
\end{array}%
\right] \in S$

Then $(S,\oplus ,.)$ is a semihyperring with hyperidentity as a null matrix
and multiplicative identity as a identity matrix.
\end{example}

\begin{example}
\label{example1}On four element semihyperring $(R,\oplus ,.)$ defined by the
following two tables:%
\begin{equation*}
\begin{tabular}{l|llll}
$\oplus $ & $0$ & $a$ & $b$ & $c$ \\ \hline
$0$ & $\{0\}$ & $\{0\}$ & $\{0\}$ & $\{0\}$ \\ 
$a$ & $\{0\}$ & $\{a,b\}$ & $\{b\}$ & $\{c\}$ \\ 
$b$ & $\{0\}$ & $\{b\}$ & $\{0,b\}$ & $\{c\}$ \\ 
$c$ & $\{0\}$ & $\{c\}$ & $\{c\}$ & $\{0,c\}$%
\end{tabular}%
\ \ \ 
\begin{tabular}{l|llll}
$\mathbf{.}$ & $0$ & $a$ & $b$ & $c$ \\ \hline
$0$ & $0$ & $0$ & $0$ & $0$ \\ 
$a$ & $0$ & $a$ & $a$ & $a$ \\ 
$b$ & $0$ & $b$ & $b$ & $b$ \\ 
$c$ & $0$ & $c$ & $c$ & $c$%
\end{tabular}%
\end{equation*}
\end{example}

By routine calculation $\{0\},$ $\{0,b\},$ $\{0,c\},$ $\{0,a,b\},$ $%
\{0,b,c\},$ $R,$ are subsemihyperrings of $R.$

\begin{definition}
By a left (right) hyperideal of $R,$ we mean a subsemihyperring $I$ of $R$
such that for all $r\in R$ and $x\in I,$ we have $r.x\in I$ $(x.r\in I)$.
\end{definition}

By a hyperideal , we mean a subsemihyperring of $R$ which is both a left and
a right hyperideal of $R,$ A hyperideal generated by a non- empty subset $A$
of a semihyperring $R$, will be denoted by $\langle A\rangle $, which is
intersection of all hyperideals of $R,$ which contains $A.$ If $I$ and $J$
be two hyperideals of a semihyperring $R,$ then the sum and product of two
hyperideals are also a hyperideal and defined as respectively: 
\begin{equation*}
I\oplus J=\underset{\QATOP{a_{i}\in I}{b_{i}\in J}}{\cup }(a_{i}\oplus b_{j})
\end{equation*}%
and%
\begin{equation*}
IJ=\{\underset{finite}{\Sigma }a_{i}b_{j};a_{i}\in I,b_{j}\in J\text{ }\}
\end{equation*}

\begin{example}
On four element semihyperring $(R,\oplus ,.)$ defined by the example \ref%
{example1}.
\end{example}

These all $\{0,b\},$ $\{0,c\},$ $\{0,a,b\},$ $\{0,b,c\},$ are right
hyperideals of semihyperring $R.$

\begin{definition}
A non-empty set $M,$ which is commutative semihypergroup with respect to
addition, with an absorbing element $0$ is called a right, $R$%
-semihypermodule $M_{R}$, if $R$ is a semihyperring and there is a function $%
\alpha :M\times R\longrightarrow P^{\ast }(M),$ where $P^{\ast
}(M)=P(M)\diagdown \{0\},$ such that if $\alpha (m,x)$ is denoted by $mx$
and $mx\subseteq M$, for all $x\in R$ and $m\in M.$ Then the following
conditions hold, for all $x,y\in R$ and $m_{1},m_{2},m\in M:$
\end{definition}

$(i)$ $(m_{1}\oplus m_{2})x=m_{1}x\oplus m_{2}x$

$(ii)$ $m(x\oplus y)=mx\oplus my$

$(iii)$ $m(xy)=(mx)y$

$(iv)$ $0.x=m.0=0.$

Similarly, we can define a left $R$-semihypermodules $_{R}M.$ A
semihyperring $R$ is a right semihypermodules over itself which will be
denoted by $R_{R}$. A non-empty subset $N$ of a right $R$-semihypermodule $M$
is called a subsemihypermodule of $M,$ if $(N,\oplus )$ is a
subsemihypergroup of $(M,\oplus )$ and $RN\subseteq P^{\ast }(N).$ Also note
that, the right (left) subsemihypermodules $R_{R}$ ($_{R}R$) are right
(left) hyperideals of $R.$

\begin{definition}
Every hyperideal of a semihyperring $R$ is a semihypermodule of $R$.
\end{definition}

Therefore by example \ref{example1} of semihyperring $R$, $\{0,b\},$ $%
\{0,c\},$ $\{0,a,b\}$ and $\{0,b,c\},$ are right $R$-semihypermodules of
semihyperring $R.$

\begin{definition}
If $X$ is a universe and $A\subseteq X,$ then characteristic function of $A$
is a function $\chi _{A}:X\rightarrow \{0,1\},$ defined by 
\begin{equation*}
\chi _{A}(x)=\QDATOPD\{ . {1\text{\ \ \ \ \ \ \ \ \ \ if }x\in A}{0\text{ \
\ \ \ \ \ \ \ if\ }x\notin A}
\end{equation*}%
A fuzzy subset $\lambda $ of $X$ is a function $\lambda :X\rightarrow \left[
0,1\right] .$ We write $\lambda (x)\in \left[ 0,1\right] ,$ for all $x\in X,$
where $\lambda $ is fuzzy subset of $X$ such that for each $x\in X$, $0\leq
\lambda (x)\leq 1.$ For any two fuzzy subsets $\lambda $ and $\mu $ of $X$, $%
\lambda \leq \mu $ if and only if $\lambda (x)\leq \mu (x),$ for all $x\in X$%
. The symbols $\lambda \wedge \mu $, and $\lambda \vee \mu $ will mean the
following fuzzy subsets of $X,$ for all $x\in X$.
\end{definition}

\begin{center}
\begin{eqnarray*}
(\lambda \wedge \mu )(x) &=&\min \{\lambda (x),\mu (x)\} \\
(\lambda \vee \mu )(x) &=&\max \{\lambda (x),\mu (x)\}
\end{eqnarray*}
\end{center}

More generally, if $\{\lambda _{i}:i\in \Lambda \}$ is a family of fuzzy
subsets of $X$, then $\wedge _{i\in \Lambda }\lambda _{i}$ and $\vee _{i\in
\Lambda }\lambda _{i}$ are defined by

\begin{center}
\begin{eqnarray*}
(\wedge _{i\in \Lambda }\lambda _{i})(x) &=&\underset{_{i\in \Lambda }}{\min 
}(\lambda _{i}(x) \\
(\vee _{i\in \Lambda }\lambda _{i})(x) &=&\underset{_{i\in \Lambda }}{\max }%
(\lambda _{i}(x)
\end{eqnarray*}
\end{center}

and will called the intersection and union of the family $\{\lambda
_{i}:i\in \Lambda \}$ of fuzzy subsets of $X.$

Let $\lambda $ be a fuzzy subset of $X$ and $t\in (0,1].$ Then the set $%
U_{t}^{\lambda }=\{x\in X:\lambda (x)\geq t\}$ is called the level subset of 
$X.$

\begin{definition}
Let $R$ be a semihyperring and $\mu $ a fuzzy set in $R.$\ Then, $\mu $ is
said to be a fuzzy hyperideal of $R$ if for all $r,x,y\in R$ the following
axioms hold:
\end{definition}

(i) $\underset{z\in x\oplus y}{\inf }\mu (z)\geq \mu (x)\wedge \mu (y),$ for
all $x,y\in R.$

(ii) $\mu (xy)\geq \mu (x)$ and $\mu (yx)\geq \mu (x)$ for all $x,y\in R.$

\begin{theorem}
Let $\mu $ be a fuzzy set in \ a semihyperring $R.$ Then, $\mu $ is a fuzzy
hyperideal of $R$ if and only if for every $t\in (0,1],$ the level subset 
\begin{equation*}
\mu _{t}=\{x\in R|\mu (x)\geq t\}\neq \varphi
\end{equation*}
\end{theorem}

is an hyperideal of $R.$

\begin{proof}
The proof is straight forward by considering the definition.
\end{proof}

\begin{definition}
Let $M$ be a right (left) $R$-semihypermodule. A function $\mu :M\rightarrow
\lbrack 0,1]$, is called a fuzzy subsemihypermodule of $M_{R}$ ($_{R}M$)$,$
if the following conditions hold for all $m_{1},m_{2},m\in M$:
\end{definition}

(i) $\mu (0_{M})=1$

(ii)$\underset{m^{\prime }\in m_{1}\oplus m_{2}}{\inf }\mu (m^{\prime })\geq
\mu (m_{1})\wedge \mu (m_{2}),$ for all $m_{1},m_{2}\in M,$

(iii) $\mu (mr)\geq \mu (m),$ $(\mu (rm)\geq \mu (m)),$ for all $r\in R$ and 
$m\in M.$

Also note that, fuzzy subsemihypermodules of $R_{R}$ ($_{R}R$) are called
fuzzy hyperideals of $R.$

Generalizing the notion of a fuzzy hypermodule \cite{3,12,13}, we formulate
the following definition.

\begin{definition}
Let $\lambda $ be a fuzzy subsemihypermodule of a right semihypermodule $%
M_{R}$ and $\mu $ a fuzzy hyperideal of $R.$
\end{definition}

Then the fuzzy subset $\lambda \mu $ of $M$ is defined by

\begin{equation*}
(\lambda \mu )(x)=\underset{x\in \Sigma _{i=1}^{p}y_{i}z_{i}}{\vee }[%
\underset{1\leq i\leq p}{\wedge }[\lambda (y_{i})\wedge \mu (z_{i})]]
\end{equation*}

where $x\in M,$ $y_{i}\in M,$ $z_{i}\in R$ and $p\in N.$

\begin{proposition}
If $\lambda $ is a fuzzy subsemihypermodule of $M_{R}$ and $\mu $ a fuzzy
hyperideal of $R.$
\end{proposition}

Then the fuzzy subset $\lambda \mu $ is a fuzzy subsemihypermodule of $M$.

\begin{proof}
We have

$(i)(\lambda \mu )(0_{M})=\underset{0\in \Sigma _{i=1}^{p}y_{i}z_{i}}{\vee }[%
\underset{1\leq i\leq p}{\wedge }[\lambda (y_{i})\wedge \mu (z_{i})]]\geq
\lambda (0_{M})\wedge \mu (0)=1.$

Thus $(\lambda \mu )(0_{M})=1.$

For $(ii)$

$(\lambda \mu )(m)=\underset{m\in \Sigma _{j=1}^{q}y_{j}^{\prime
}z_{j}^{\prime }}{\vee }[\underset{1\leq j\leq q}{\wedge }[\lambda
(y_{j}^{\prime })\wedge \mu (z_{j}^{\prime })]],$

and`

$(\lambda \mu )(m^{\prime })=\underset{m^{\prime }\in \Sigma
_{k=1}^{r}y_{k}^{^{\prime \prime }}z_{k}^{^{\prime \prime }}}{\vee }[%
\underset{1\leq k\leq r}{\wedge }[\lambda (y_{k}^{^{\prime \prime }})\wedge
\mu (z_{k}^{^{\prime \prime }})]],$ where $m,m^{\prime }\in M.$

Thus

$(\lambda \mu )(m)\wedge (\lambda \mu )(m^{\prime })$

$=[\underset{m\in \Sigma _{j=1}^{q}y_{j}^{\prime }z_{j}^{\prime }}{\vee }[%
\underset{1\leq j\leq q}{\wedge }[\lambda (y_{j}^{\prime })\wedge \mu
(z_{j}^{\prime })]]]\wedge \lbrack \underset{m^{\prime }\in \Sigma
_{k=1}^{r}y_{k}^{^{\prime \prime }}z_{k}^{^{\prime \prime }}}{\vee }[%
\underset{1\leq k\leq r}{\wedge }[\lambda (y_{k}^{^{\prime \prime }})\wedge
\mu (z_{k}^{^{\prime \prime }})]]]$

\ \ \ \ \ \ \ \ \ \ \ \ \ \ \ \ \ \ \ \ \ \ \ \ \ \ \ \ \ \ \ \ \ (using the
infinite meet distributive law)

$=\underset{m\in \Sigma _{j=1}^{q}y_{j}^{\prime }z_{j}^{\prime }}{\vee }$ $%
\underset{m^{\prime }\in \Sigma _{k=1}^{r}y_{k}^{^{\prime \prime
}}z_{k}^{^{\prime \prime }}}{\vee }[\underset{1\leq j\leq q}{\wedge }%
(\lambda (y_{j}^{\prime })\wedge \mu (z_{j}^{\prime }))]\wedge \lbrack 
\underset{1\leq k\leq r}{\wedge }(\lambda (y_{k}^{^{\prime \prime }})\wedge
\mu (z_{k}^{^{\prime \prime }}))]$

$\leq \underset{m\oplus m^{\prime }\subseteq \Sigma
_{l=1}^{s}y_{l}^{^{\prime \prime \prime }}z_{l}^{^{\prime \prime \prime }}}{%
\vee }[\underset{1\leq l\leq s}{\wedge }(\lambda (y_{l}^{^{\prime \prime
\prime }})\wedge \mu (z_{l}^{^{^{\prime \prime \prime }}}))]=\underset{m"\in
m\oplus m^{\prime }}{\inf }\lambda \mu (m^{\prime \prime })$

and for $(iii)$

$(\lambda \mu )(m)=\underset{m\in \Sigma _{j=1}^{q}y_{j}^{\prime
}z_{j}^{\prime }}{\vee }[\underset{1\leq j\leq q}{\wedge }[\lambda
(y_{j}^{\prime })\wedge \mu (z_{j}^{\prime })]],$

$\leq \underset{m\in \Sigma _{j=1}^{q}y_{j}^{\prime }z_{j}^{\prime }}{\vee }[%
\underset{1\leq j\leq q}{\wedge }[\lambda (y_{j}^{\prime })\wedge \mu
(z_{j}^{\prime }a)]],$ where $a$ is any element of $R.$

$\leq \underset{ma\subseteq \Sigma _{j=1}^{t}y_{j}^{^{\prime \prime
}}z_{j}^{^{\prime \prime }}}{\vee }[\underset{1\leq n\leq t}{\wedge }%
[\lambda (y_{n}^{^{\prime \prime }})\wedge \mu (z_{n}^{^{^{\prime \prime
}}})]]=\lambda \mu (ma).$
\end{proof}

\begin{corollary}
If $\lambda $ and $\mu $ are fuzzy hyperideals of $R,$ then $\lambda \mu $
is a fuzzy hyperideal of $R,$ called the product of $\lambda $ and $\mu .$
\end{corollary}

\begin{remark}
If $\lambda $ and $\mu $ are fuzzy hyperideals of $R,$ then $\lambda \wedge
\mu $ is clearly a fuzzy hyperideal of $R.$ In general, $\lambda \wedge \mu
\neq \lambda \mu .$
\end{remark}

\begin{definition}
If $\lambda $ and $\mu $ are fuzzy hyperideals of $R.$ The fuzzy subset $%
\lambda \oplus \mu $ of $R$ is defined by
\end{definition}

$(\lambda \oplus \mu )(x)=\underset{x\in y\oplus z}{\vee }[\lambda (y)\wedge
\mu (z)],$ for $x\in R.$

\begin{proposition}
For fuzzy hyperideals $\lambda $ and $\mu $ of $R,$ $\lambda \oplus \mu $ is
a fuzzy hyperideal of $R$ (called the sum of $\lambda $ and $\mu $).
\end{proposition}

\begin{proof}
For \ any $x,$ $x^{\prime }\in R$

$(\lambda \oplus \mu )(x)\wedge (\lambda \oplus \mu )(x^{\prime })=[\underset%
{x\in y\oplus z}{\vee }(\lambda (y)\wedge \mu (z))]\wedge \lbrack \underset{%
x^{\prime }\in y^{\prime }\oplus z^{\prime }}{\vee }(\lambda (y^{\prime
})\wedge \mu (z^{\prime }))]$

$=\underset{\QATOP{x\in y\oplus z}{x^{\prime }\in y^{\prime }\oplus
z^{\prime }}}{\vee }[[(\lambda (y)\wedge \mu (z))]\wedge \lbrack \lambda
(y^{\prime })\wedge \mu (z^{\prime })]]$

$=\underset{\QATOP{x\in y\oplus z}{x^{\prime }\in y^{\prime }\oplus
z^{\prime }}}{\vee }[[(\lambda (y)\wedge \lambda (y^{\prime }))]\wedge
\lbrack \mu (z)\wedge \mu (z^{\prime })]]$

$\leq \underset{x"\in y"\oplus z"}{\vee }[\underset{y"\in y\oplus y^{\prime }%
}{\inf }\lambda (y")\wedge \underset{z"\in z\oplus z^{\prime }}{\inf }\mu
(z")]$

$\leq \underset{x"\in x\oplus x^{\prime }}{\inf }(\lambda \oplus \mu )(x").$

Again

$(\lambda \oplus \mu )(x)=\underset{x\in y\oplus z}{\vee }[(\lambda
(y)\wedge \mu (z))]$

$\leq \underset{xa\subseteq ya\oplus za}{\vee }[\lambda (ya)\wedge \mu (za)]$
\ (where $a$ is any element of $R$)

$\leq \underset{xa\subseteq y^{\prime }\oplus z^{\prime }}{\vee }[\lambda
(y^{\prime })\wedge \mu (z^{\prime })]$

$=\underset{b\in xa}{\inf }(\lambda \oplus \mu )(b)$

Hence $\lambda \oplus \mu $ is a fuzzy hyperideal of $R.$
\end{proof}

\section{\protect\huge Fully idempotent semihyperrings}

A semihyperring $R$ is called fully idempotent if each hyperideal of $R$ is
idempotent (a hyperideal $I$ is idempotent if $I^{2}=I$), and a
semihyperring $R$ is said to be regular if for each $x\in R$, there exist $%
a\in R$ such that $x=xax$.

\begin{lemma}
A semihyperring $R$ is regular if and only if for any right hyperideal $I$
and for any left hyperideal $L$ of $R,$ we have $IL=I\cap L$.
\end{lemma}

Concerning these semihyperrings, we prove the following characterization
theorem.

\begin{theorem}
\label{THR 2.1}The following conditions for a semihyperring $R,$ are
equivalent :
\end{theorem}

(1) $R$ is fully idempotent,

(2) Fuzzy hyperideal of $R$ is idempotent.

(3) For each pair of fuzzy hyperideals $\lambda $ and $\mu $ of $R,$ $%
\lambda \wedge \mu =\lambda \mu .$

If $R$ is assumed to be commutative (that is, $xy=yx$ for all $x,y\in R$),
then the above conditions are equivalent to:

(4) $R$ is regular.

\begin{proof}
(1)$\Rightarrow $(2). Let $\delta $ be a fuzzy hyperideal of $R.$ For any $%
x\in R,$

$\delta ^{2}(x)=(\delta \delta )(x)$

$=\underset{x\in \Sigma _{i=1}^{p}y_{i}z_{i}}{\vee }[\underset{1\leq i\leq p}%
{\wedge }(\delta (y_{i})\wedge \delta (z_{i}))]$

$\leq \underset{x\in \Sigma _{i=1}^{p}y_{i}z_{i}}{\vee }[\underset{1\leq
i\leq p}{\wedge }(\delta (y_{i}z_{i})\wedge \delta (y_{i}z_{i}))]$

$=\underset{x\in \Sigma _{i=1}^{p}y_{i}z_{i}}{\vee }[[\underset{1\leq i\leq p%
}{\wedge }\delta (y_{i}z_{i})]\wedge \lbrack \underset{1\leq i\leq p}{\wedge 
}\delta (y_{i}z_{i})]]$

$\leq \underset{x\in \Sigma _{i=1}^{p}y_{i}z_{i}}{\vee }[\delta (x)\wedge
\delta (x)]=\delta (x).$

Since each hyperideal of $R$ is idempotent, therefore, $(x)=(x)^{2},$ for
each $x\in R.$ Since $x\in (x)$ it follows that $x\in (x)^{2}=RxRRxR.$
Hence, $x=\Sigma _{i=1}^{q}a_{i}xa_{i}^{\prime }b_{i}xb_{i}^{\prime }$ where 
$a_{i},a_{i}^{\prime },b_{i},b_{i}^{\prime }\in R$ and $q\in N.$ Now, $%
\delta (x)=\delta (x)\wedge \delta (x)\leq \delta (a_{i}xa_{i}^{\prime
})\wedge \delta (b_{i}xb_{i}^{\prime })$ \ \ $(1\leq i\leq q).$

Therefore,

$\delta (x)\leq \underset{1\leq i\leq q}{\wedge }[\delta
(a_{i}xa_{i}^{\prime })\wedge \delta (b_{i}xb_{i}^{\prime })]$

$\leq \underset{x\in \Sigma _{i=1}^{q}a_{i}xa_{i}^{\prime
}b_{i}xb_{i}^{\prime }}{\vee }[\underset{1\leq i\leq q}{\wedge }[\delta
(a_{i}xa_{i}^{\prime })\wedge \delta (b_{i}xb_{i}^{\prime })]]$

$\leq \underset{x\in \Sigma _{i=1}^{q}y_{i}z_{i}}{\vee }[\underset{1\leq
j\leq r}{\wedge }[\delta (y_{i})\wedge \delta (z_{i})]]$

$=(\delta \delta )(x)=\delta ^{2}(x).$

Thus $\delta ^{2}=\delta .$

(2)$\Rightarrow $(1). Let $I$ be an hyperideal of $R.$ Thus $\delta _{I},$
the characteristic function of $I,$ is a fuzzy hyperideal of $R.$ Hence $%
\delta _{I}^{2}=\delta _{I}.$ Therefore, $\delta _{I}\delta _{I}=\delta
_{I}, $ hence $\delta _{I^{2}}=\delta _{I}.$ It follows that $I^{2}=I.$
Hence (2)$\Leftrightarrow $(1).

(1)$\Rightarrow $(3). Let $\lambda $ and $\mu $ be any pair of fuzzy
hyperideals of $R.$ For any $x\in R.$

$(\lambda \mu )(x)=\underset{x\in \Sigma _{i=1}^{p}y_{i}z_{i}}{\vee }[%
\underset{1\leq i\leq p}{\wedge }(\lambda (y_{i})\wedge \mu (z_{i}))]$

$\leq \underset{x\in \Sigma _{i=1}^{p}y_{i}z_{i}}{\vee }[\underset{1\leq
i\leq p}{\wedge }(\lambda (y_{i}z_{i})\wedge \mu (y_{i}z_{i}))]$

$\leq \underset{x\in \Sigma _{i=1}^{p}y_{i}z_{i}}{\vee }[[\underset{1\leq
i\leq p}{\wedge }\lambda (y_{i}z_{i})]\wedge \lbrack \underset{1\leq i\leq p}%
{\wedge }\mu (y_{i}z_{i})]]$

$\leq \underset{x\in \Sigma _{i=1}^{p}y_{i}z_{i}}{\vee }[\lambda (x)\wedge
\mu (x)]$

$=\lambda (x)\wedge \mu (x)=(\lambda \wedge \mu )(x).$

Again, since $R$ is fully idempotent, $(x)=(x)^{2},$ for any $x\in R.$
Hence, as argued in the first part of the proof of this theorem, we have

$(\lambda \wedge \mu )(x)=\lambda (x)\wedge \mu (x)$

$\leq \underset{x\in \Sigma _{i=1}^{p}y_{i}z_{i}}{\vee }[\underset{1\leq
i\leq p}{\wedge }(\lambda (y_{i})\wedge \mu (z_{i}))]$

$=(\lambda \mu )(x).$

Thus $\lambda \wedge \mu =\lambda \mu .$

(3)$\Rightarrow $(1). Let $\lambda $ and $\mu $ be any pair of fuzzy
hyperideals of $R.$ We have, $\lambda \wedge \mu =\lambda \mu .$ take $\mu
=\lambda .$

Thus $\lambda \wedge \lambda =\lambda ^{2},$ that is, $\lambda =\lambda
^{2}, $ where $\lambda $ is any fuzzy hyperideal of $R.$ Hence, (3)$%
\Rightarrow $(2). Since we already proved that (1) and (2) are equivalent,
hence it follows that (3)$\Rightarrow $(1) and (1)$\Rightarrow $(3). This
establishes (1)$\Leftrightarrow $(2)$\Leftrightarrow $(3). Finally, if $R$
is commutative then it is easy to verify that (1)$\Leftrightarrow $(4).

Next, we prove another characterization theorem for fully idempotent
semihyperrings.
\end{proof}

\begin{theorem}
\label{THM1}The following conditions for a semihyperring $R,$ are equivalent;
\end{theorem}

(1) $R$ is fully idempotent,

(2) The set of all fuzzy hyperideals of $R$ (ordered by $\leq $ ) forms a
distributive lattice $FI_{R}$ under the sum and intersection of fuzzy
hyperideals with $\lambda \wedge \mu =\lambda \mu ,$ for each pair of fuzzy
hyperideals $\lambda ,\mu $ of $R.$

\begin{proof}
(1)$\Leftrightarrow $(2). The set $FI_{R}$ of all fuzzy hyperideals of $R$
(ordered by $\leq $ ) is clearly a lattice under the sum and intersection of
fuzzy hyperideals. Moreover, since $R$ is a fully idempotent semihyperring,
it follows from above Theorem that $\lambda \wedge \mu =\lambda \mu ,$ for
each pair of fuzzy hyperideals $\lambda ,\mu $ of $R.$ We now show that $%
FI_{R}$ is a distributive lattice, that is, for fuzzy hyperideals $\lambda
,\delta $ and $\eta $ of $R,$ we have

$[(\lambda \wedge \delta )\oplus \eta )]=[(\lambda \oplus \eta )\wedge
(\delta \oplus \eta )].$

For any $x\in R,$

$[(\lambda \wedge \delta )\oplus \eta )](x)=\underset{x\in y\oplus z}{\vee }%
[(\lambda \wedge \delta )(y)\wedge \eta (z)]$

$=\underset{x\in y\oplus z}{\vee }[\lambda (y)\wedge \delta (y)\wedge \eta
(z)]$

$=\underset{x\in y\oplus z}{\vee }[\lambda (y)\wedge \delta (y)\wedge \eta
(z)\wedge \eta (z)]$

$=\underset{x\in y\oplus z}{\vee }[\lambda (y)\wedge \eta (z)]\wedge \lbrack
\delta (y)\wedge \eta (z)]$

$=\underset{x\in y\oplus z}{\vee }[(\lambda \oplus \eta )(x)\wedge (\delta
\oplus \eta )(x)]$

because, for $x\in y\oplus z,$ $\lambda (y)\oplus \eta (z)\leq (\lambda
\oplus \eta )(x)$ and, similarly, $\delta (y)\wedge \eta (z)\leq (\delta
\oplus \eta )(x)$

$=(\lambda \oplus \eta )(x)\wedge (\delta \oplus \eta )(x)$

$=[(\lambda \oplus \eta )(x)\wedge (\delta \oplus \eta )](x)$

Again,

$[(\lambda \oplus \eta )\wedge (\delta \oplus \eta )](x)$

$=[(\lambda \oplus \eta )(\delta \oplus \eta )](x)$

$=\underset{x\in \Sigma _{i=1}^{p}y_{i}z_{i}}{\vee }[\underset{1\leq i\leq p}%
{\wedge }[(\lambda \oplus \eta )(y_{i})\wedge (\delta \oplus \eta )(z_{i})]]$

$=\underset{x\in \Sigma _{i=1}^{p}y_{i}z_{i}}{\vee }[\underset{1\leq i\leq p}%
{\wedge }[[\underset{y_{i}\in r_{i}\oplus s_{i}}{\vee }(\lambda
(r_{i})\wedge \eta (s_{i}))]\wedge \lbrack \underset{z_{i}\in t_{i}\oplus
u_{i}}{\vee }(\delta (t_{i})\wedge \eta (u_{i})]]]$

$=\underset{x\in \Sigma _{i=1}^{p}y_{i}z_{i}}{\vee }[\underset{1\leq i\leq p}%
{\wedge }[\underset{\QATOP{y_{i}\in r_{i}\oplus s_{i}}{z_{i}\in t_{i}\oplus
u_{i}}}{\vee }[(\lambda (r_{i})\wedge \eta (s_{i})\wedge (\delta
(t_{i})\wedge \eta (u_{i}))]]]$

using infinite meet distributive law

$=\underset{x\in \Sigma _{i=1}^{p}y_{i}z_{i}}{\vee }[\underset{1\leq i\leq p}%
{\wedge }[\underset{\QATOP{y_{i}\in r_{i}\oplus s_{i}}{z_{i}\in t_{i}\oplus
u_{i}}}{\vee }[(\lambda (r_{i})\wedge \eta (s_{i})\wedge \eta (s_{i})\wedge
(\delta (t_{i})\wedge \eta (u_{i})]]]$

$\leq \underset{x\in \Sigma _{i=1}^{p}y_{i}z_{i}}{\vee }[\underset{1\leq
i\leq p}{\wedge }[\underset{\QATOP{y_{i}\in r_{i}\oplus s_{i}}{z_{i}\in
t_{i}\oplus u_{i}}}{\vee }[(\lambda (r_{i}t_{i})\wedge \delta
(r_{i}t_{i})\wedge \eta (s_{i}t_{i})\wedge \eta (s_{i}u_{i})\wedge \eta
(r_{i}u_{i})]]]$

$\leq \underset{x\in \Sigma _{i=1}^{p}y_{i}z_{i}}{\vee }[\underset{1\leq
i\leq p}{\wedge }[\underset{\QATOP{y_{i}\in r_{i}\oplus s_{i}}{z_{i}\in
t_{i}\oplus u_{i}}}{\vee }[(\lambda \wedge \delta )(r_{i}t_{i})\wedge \eta
(s_{i}t_{i}\oplus s_{i}u_{i}\oplus r_{i}u_{i})]]]$

$\leq \underset{x\in \Sigma _{i=1}^{p}y_{i}z_{i}}{\vee }[\underset{1\leq
i\leq p}{\wedge }[(\lambda \wedge \delta )\wedge \eta ](y_{i}z_{i})]$

$\leq \underset{x\in \Sigma _{i=1}^{p}y_{i}z_{i}}{\vee }[(\lambda \wedge
\delta )\wedge \eta ](x)]$

$=[(\lambda \wedge \delta )\wedge \eta ](x).$

$[(\lambda \oplus \eta )\wedge (\delta \oplus \eta )]=[(\lambda \wedge
\delta )\wedge \eta ].$

(2)$\Leftrightarrow $(1). Suppose that the set $FI_{R},$ of all the fuzzy
hyperideals of $R$ (ordered by $\leq $ ) is a distributive lattice under the
sum and intersection of fuzzy hyperideals with $\lambda \wedge \mu =\lambda
\mu ,$ for each pair of fuzzy hyperideals $\lambda ,\mu $ of $R.$

Then for any fuzzy hyperideals $\lambda $ of $R,$ we have $\lambda
^{2}=\lambda .\lambda =\lambda \wedge \lambda =$g.l.b$.$ of $\{\lambda
,\lambda \}=\lambda .$ Hence $R$ is fully idempotent.
\end{proof}

\begin{definition}
A hyperideal $I$ of a semihyperring $R$ is called a prime hyperideal of $R$
if for all hyperideals $A,B$ of $R,$ $AB\subseteq I$ implies that either $%
A\subseteq I$ or $B\subseteq I$.
\end{definition}

\begin{definition}
A hyperideal $I$ of a semihyperring $R$ is called an irreducible if for all
hyperideals $A,B$ of $R,$ $A\cap B=I,$ implies $A=I$ or $B=I.$
\end{definition}

\begin{definition}
A fuzzy hyperideal $\eta $ of a semihyperring $R$ is called a fuzzy prime
hyperideal of $R$ if for\ fuzzy hyperideals $\lambda $ and $\mu $ of $R,$ $%
\lambda \mu \leq \eta \Rightarrow \lambda \leq \eta $ or $\mu \leq \eta ;$ $%
\eta $ is called fuzzy irreducible if for fuzzy hyperideals $\lambda ,\mu $
of $\lambda $,$\mu $ of $R,$ $\lambda \wedge \mu =\eta \Rightarrow \lambda
=\eta $ or $\mu =\eta .$
\end{definition}

\begin{theorem}
\label{THR3}Let $R$ be a fully idempotent semihyperring. For a fuzzy
hyperideal $\eta $ of $R,$ the following conditions are equivalent:
\end{theorem}

(1) $\eta $ is a fuzzy prime hyperideal,

(2) $\eta $ is a fuzzy irreducible hyperideal.

\begin{proof}
(1) Assume that $\eta $ is a fuzzy prime hyperideal. We show that $\eta $ is
fuzzy irreducible, that is , for fuzzy hyperideals $\lambda ,\mu $ of $R,$ $%
\lambda \wedge \mu =\eta \Rightarrow \lambda =\eta $ or $\mu =\eta .$ Since $%
R$ is a fully idempotent semihyperring, the set of fuzzy hyperideals of $R$
(ordered by $\leq $ ) is a distributive lattice under the sum and
intersection of fuzzy hyperideals by \ref{THM1}

This implies that $\eta =g.l.b.$of \{$\lambda ,\mu $\}, since $\eta =\lambda
\wedge \mu .$ Thus it follows that $\lambda \leq \eta $ and $\eta \leq \mu .$
On the other hand, as $\lambda \leq R$ is fully idempotent, it follows from %
\ref{THR 2.1} (3) that $\lambda \wedge \mu =\lambda \mu .$ Hence $\eta
=\lambda \wedge \mu =\lambda \mu $. Since $\eta $ is a fuzzy prime
hyperideal, by the above definition, either $\lambda \leq \eta $ or $\mu
\leq \eta .$ As already noted, $\eta \leq \lambda $ and $\eta \leq \mu ;$ so
it follows that either $\lambda =\eta $ or $\mu =\eta .$ Hence $\eta $ is a
fuzzy irreducible hyperideal.

(2) Conversely, assume that $\eta $ is a fuzzy irreducible hyperideal. We
show that $\eta $ is a fuzzy prime hyperideal. Suppose there exist fuzzy
hyperideals $\lambda $ and $\mu $ such that $\lambda \mu \leq \eta .$ Since $%
R$ is assumed to be a fully idempotent semihyperring, it follows from \ref%
{THR 2.1}(3) that the set of fuzzy hyperideals of $R$ (ordered by $\leq $ )
is a distributive lattice with respect to the sum and intersection of fuzzy
hyperideals. Hence the inequality $\lambda \wedge \mu \leq \eta \Rightarrow
(\lambda \wedge \mu )\oplus \eta =\eta ,$ and using the distributivity of
this lattice, we have $\eta =(\lambda \wedge \mu )\oplus \eta =(\lambda
\oplus \eta )\wedge (\mu \oplus \eta )].$ Since $\eta $ is fuzzy
irreducible, it follows that either $\lambda \oplus \eta =\eta $ or $\mu
\oplus \eta =\eta .$ This implies that either $\lambda \leq \eta $ or $\mu
\leq \eta .$ Hence $\eta $ is a fuzzy prime hyperideal.
\end{proof}

\begin{lemma}
\label{LEMMA1}Let $R$ be a fully idempotent semihyperring. If $\lambda $ is
a fuzzy hyperideal of $R$ with $\lambda (a)=\alpha $, where $a$ is any
element of $R$ and $\alpha \in \lbrack 0,1],$ then there exists a fuzzy
prime hyperideal $\eta $ of $R$ such that $\lambda \leq \eta $ and $\eta
(a)=\alpha .$
\end{lemma}

\begin{proof}
Let $X=\{\mu :\mu $ is a fuzzy hyperideal of $R,$ $\mu (a)=\alpha ,$ and $%
\lambda \leq \mu \}.$ Then $X\neq \phi ,$ since $\lambda \in X.$ Let $\tau $
be a totally ordered subset of $X$, say $\tau =\{\lambda _{i}:i\in I\}.$ We
claim that $\vee _{i\in I}\lambda _{i}$ is a fuzzy hyperideal of $R.$
Clearly, $\vee _{i\in I}\lambda _{i}(x)=1.$ Also, for any $x,r\in R,$ we have

$\underset{i}{\vee }\lambda _{i}(x)=\underset{i}{\vee }\lambda _{i}(x)\leq $ 
$[\vee (\lambda _{i}(xr))]=\underset{i}{\vee }\lambda _{i}(xr).$

Similarly, $\underset{i}{\vee }\lambda _{i}(x)\leq $ $\underset{i}{\vee
\lambda _{i}(rx)}.$ Finally, we show that $\underset{z\in x\oplus y}{\inf }%
\underset{i}{\vee }\lambda _{i}(z)\geq \underset{i}{\vee }\lambda
_{i}(x)\wedge \underset{i}{\vee }\lambda _{i}(y),$ for any $x,y\in R.$
Consider

$\underset{i\in I}{\vee }\lambda _{i}(x)\wedge \underset{i\in I}{\vee }%
\lambda _{i}(y)=\underset{i\in I}{\vee }\lambda _{i}(x)\wedge \underset{j\in
I}{\vee }\lambda _{j}(y)$

$=[\underset{i\in I}{\vee }\lambda _{i}(x)]\wedge \underset{j\in I}{\vee }%
[\lambda _{j}(y)]$

$=\underset{j\in I}{\vee }[[\underset{i\in I}{\vee }\lambda _{i}(x)]\wedge
\lambda _{j}(y)]$

$=\underset{j\in I}{\vee }[\underset{i\in I}{\vee }[\lambda _{i}(x)\wedge
\lambda _{j}(y)]]$

$\leq \underset{j\in I}{\vee }[\underset{i\in I}{\vee }[\lambda
_{i}^{j}(x)\wedge \lambda _{i}^{j}(y)]]$ where $\lambda _{i}^{j}\in
\{\lambda _{i}:i\in I\}$

$\leq \underset{j\in I}{\vee }[\underset{i\in I}{\vee }[\underset{z\in
x\oplus y}{\inf }\lambda _{i}^{j}(z)]]$

$\leq \underset{i,j\in I}{\vee }[\underset{z\in x\oplus y}{\inf }\lambda
_{i}^{j}(z)]$

$\leq \underset{i\in I}{\vee }[\underset{z\in x\oplus y}{\inf }\lambda
_{i}^{j}(z)]$

$=\underset{i\in I}{\vee }\underset{z\in x\oplus y}{\inf }\lambda
_{i}^{j}(z).$

Thus $\underset{i\in I}{\vee }$ is a fuzzy hyperideal of $R.$ Clearly $%
\lambda \leq \underset{i\in I}{\vee }\lambda _{i}$ and $\underset{i\in I}{%
\vee }\lambda _{i}(a)=\underset{i\in I}{\vee }\lambda _{i}(a)=\alpha .$ Thus 
$\underset{i\in I}{\vee }\lambda _{i}$ is the l.u.b of $\tau .$

Hence, by Zorn's lemma, there exists a fuzzy hyperideal $\eta $ of $R$ which
is maximal with respect to the property that $\lambda \leq \eta $ and $\eta
(a)=\alpha .$ We now show that $\eta $ is a fuzzy irreducible hyperideal of $%
R.$ Suppose $\eta =\delta _{1}\wedge \delta _{2},$ where $\delta _{1}$ and $%
\delta _{2}$ are fuzzy hyperideal of $R.$ Since $R$ is assumed to be a fully
idempotent semihyperring, so by Theorem \ref{THM1}, the set of fuzzy
hyperideals of $R$ (ordered by $\leq $ ) is a distributive lattice under the
sum and intersection of fuzzy hyperideals. Hence $\eta =\delta _{1}\wedge
\delta _{2}=g.l.b.\{\delta _{1},\delta _{2}\}.$ This implies that $\eta \leq
\delta _{1}$ and $\eta \leq \delta _{2}.$ We claim that either $\eta =\delta
_{1}$ or $\eta =\delta _{2}.$ Suppose, on the contrary, $\eta \neq \delta
_{1}$ and $\eta \neq \delta _{2},$ it follows that $\delta _{1}(a)\neq
\alpha $ and $\delta _{2}(a)\neq \alpha $. Hence $\alpha =\eta (a)=(\delta
_{1}\wedge \delta _{2})(a)=\{\delta _{1}(a)\wedge \delta _{2}(a)\}\neq
\alpha ,$ which is contradiction. Hence either $\eta =\delta _{1}$ or $\eta
=\delta _{2}.$This proves that $\eta $ is a fuzzy irreducible hyperideal.
Hence by Theorem\ref{THR3}, $\eta $ is a fuzzy prime hyperideal.
\end{proof}

Now, we have to prove the main characterization theorem for fully idempotent
semihyperrings.

\begin{theorem}
The following conditions for a semihyperring $R$ are equivalent:
\end{theorem}

(1) $R$ is fully idempotent,

(2) The lattice of all fuzzy hyperideals of $R$ (ordered by $\leq $ ) is a
distributive lattice under the sum and intersection of fuzzy hyperideals
with $\lambda \wedge \mu =\lambda \mu ,$ for each pair of fuzzy hyperideals $%
\lambda ,\mu $ of $R.$

(3) Each fuzzy hyperideal is the intersection of those fuzzy prime
hyperideals of $R$ which contain it. If, in addition, $R$ is assumed to be
commutative, then the above conditions are equivalent to:

(4) $R$ is regular.

\begin{proof}
(1)$\Rightarrow $(2). This is Theorem \ref{THM1}

(2)$\Rightarrow $(3). Let $\lambda $ be a fuzzy hyperideal of $R$ and let $%
\{\lambda _{s}:s\in \Omega \}$ be the family of all fuzzy prime hyperideals
of $R$ which contain $\lambda .$ Obviously, $\lambda \leq \wedge _{s\in
\Omega }\lambda _{s}.$ We now prove that $\wedge _{s\in \Omega }\lambda
_{s}\leq \lambda .$ Let $a$ be any element of $R.$ By Lemma \ref{LEMMA1},
there exists a fuzzy prime hyperideal $\lambda _{t}$ $($say$)$ such that $%
\lambda \leq \lambda _{t}$ and $\lambda (a)=\lambda _{t}(a).$ Thus $\lambda
_{t}\in \{\lambda _{s}:s\in \Omega \}.$ Hence $\wedge _{s\in \Omega }\lambda
_{s}\leq \lambda _{t},$ so $\wedge _{s\in \Omega }\lambda _{s}(a)\leq
\lambda _{t}(a)=\lambda (a).$ This implies that $\wedge _{s\in \Omega
}\lambda _{s}\leq \lambda ,$ so $\wedge _{s\in \Omega }\lambda _{s}=\lambda
. $

(3)$\Rightarrow $(1). Let $\lambda $ be any fuzzy hyperideal of $R$. Then $%
\lambda ^{2}$ is also a fuzzy hyperideal of $R.$ Hence, according to
statement (3), $\lambda ^{2}$ can be written as $\lambda ^{2}=\wedge _{s\in
\Omega }\lambda _{s},$ where $\{\lambda _{s}:s\in \Omega \}$ is the family
of all fuzzy prime hyperideals of $R$ which contains $\lambda ^{2}.$ Now $%
\lambda ^{2}\leq \lambda $ is always true. Hence, $\lambda ^{2}=\lambda .$
Therefore, $R$ is fully idempotent. Finally, if $R$ is assumed to be
commutaive, then as noted in Theorem \ref{THR 2.1}, (1)$\Leftrightarrow $%
(4). This completes the proof of the theorem.
\end{proof}

At the end of this section, we prove the following fuzzy theoretic
characterization of regular semihyperring. First we recall the following
definition.

\begin{definition}
Let $\lambda $ and $\mu $ be fuzzy subsets of a semihyperring $R.$ Then the
fuzzy subset $\lambda \circ \mu $ is defined by $(\lambda \circ \mu )(x)=%
\underset{x=yz}{\vee }[(\lambda (y)\wedge \mu (z))],$ for all $x\in R.$
\end{definition}

\begin{theorem}
The following conditions for a semihyperring $R$ are equivalent:
\end{theorem}

(1) $R$ is regular,

(2) For any right hyperideal of $R$ and any left hyperideal $L$ of $R,$ $%
R\cap L=RL,$

(3) For any fuzzy right hyperideal $\lambda $ and any fuzzy left hyperideal $%
\mu $ of $R,$

$\lambda \wedge \mu =\lambda \circ \mu .$

\begin{proof}
For (1)$\Rightarrow $(2), we refer to Golan [4, proposition 5.27, p.63]. So
we have to prove only (1)$\Rightarrow $(3). Suppose that $R$ is regular. Let 
$\delta $ be any fuzzy right hyperideal and $\mu $ any fuzzy left hyperideal
of $R.$ We show that $\lambda \wedge \mu =\lambda \circ \mu .$ Let $x\in R.$
Then

$(\lambda \circ \mu )(x)=\underset{x=yz}{\vee }[\lambda (y)\wedge \mu (z)]$

$\leq \underset{x=yz}{\vee }[\lambda (yz)\wedge \mu (yz)]=\underset{x=yz}{%
\vee }[\lambda (x)\wedge \mu (x)]$

$=\underset{x=yz}{\vee }[\lambda (x)\wedge \mu (x)]=(\lambda \wedge \mu
)(x). $

Thus $(\lambda \circ \mu )\leq (\lambda \wedge \mu ).$ This does not depend
upon the hypothesis. We now show that $(\lambda \wedge \mu )\leq (\lambda
\circ \mu ).$ Let $x\in R.$\ Since $R$ is von Neumann regular, there exists $%
a\in R$ such that $x=xax.$ Thus

$(\lambda \wedge \mu )(x)=(\lambda (x)\wedge \mu (x))\leq (\lambda
(xa)\wedge \mu (x)\leq \underset{x=yz}{\vee }(\lambda (y)\wedge \mu
(z))=(\lambda \circ \mu )(x).$

Hence $\lambda \circ \mu =\lambda \wedge \mu .$ Conversely, assume that $%
\lambda \wedge \mu =\lambda \circ \mu $ for any fuzzy right hyperideal $%
\lambda $ and any left hyperideal $\mu $ of $R.$ We show that $R$ is
regular. Let $x\in R$. $xR$ and $Rx$ are the principal right and left
hyperideals of $R,$ respectively, which are generated by $x.$ Thus, if $%
\delta _{xR}$ and $\delta _{Rx}$ denote, respectively, the characteristic
functions of $xR$ and $Rx,$ then clearly $\delta _{xR}$ and $\delta _{Rx}$
are fuzzy right and left hyperideals of $R.$ Hence, by the assumption $%
\delta _{xR}\wedge \delta _{Rx}=\delta _{xR}\circ \delta _{Rx}.$ This
implies that $xR\cap Rx=xRRx.$ Thus $x\in xR\cap Rx=xRRx\subseteq xRx.$
Hence, there exists $x\in R$ such that $x=xax,$ thus showing that $R$ is
regular.
\end{proof}

\section{\protect\huge Fuzzy prime spectrum of a fully idempotent
semihyperring}

In this section $R$ will denote a fully idempotent semihyperring, $FI_{R}$
will denote the lattice of fuzzy hyperideals of $R,$ and $FP_{R}$ the set of
all proper fuzzy prime hyperideals of $R.$ For any fuzzy hyperideal $\lambda 
$ of $R,$ we define $O_{\lambda }=\{\mu \in FP_{R}:\lambda \nleq \mu \}$ and 
$\tau (FP_{R})=\{O_{\lambda }:\lambda \in FI_{R}\}.$ A fuzzy hyperideal $%
\lambda $ of $R$ is called proper if $\lambda \neq A,$ where the fuzzy
hyperideal $A$ of $R$ is defined by $A(x)=1,$ for all $x\in R.$ We prove the
following.

\begin{theorem}
The set $\tau (FP_{R})$ forms a topology on the set $FP_{R}.$ The assignment 
$\lambda \longmapsto O_{\lambda }$ is an isomorphism between the lattice $%
FI_{R}$ of fuzzy hyperideals of $R$ and the lattice of open subsets of $%
FP_{R}.$
\end{theorem}

\begin{proof}
First we show that $\tau (FP_{R})$ forms a topology on the set $FP_{R}.$
Note that $O_{\varphi }=\{\mu \in FP_{R}:\varphi \nleq \mu \}=\varphi ,$
where $\varphi $ is the usual empty set and $\varphi $ is the fuzzy zero
hyperideal of $R$ defined by $\varphi (a)=0$ for all $a\in R.$ This follows
since $\varphi $ is contained in every fuzzy (prime) hyperideal of $R.$ Thus 
$O_{\varphi }$ is the empty subset of $\tau (FP_{R}).$ On the other hand, we
have $O_{A}=\{\mu \in FP_{R}:A\nleq \mu \}=FP_{R}.$ This is true, since $%
FP_{R}$ is the set of proper fuzzy prime hyperideals of $R.$ So $%
O_{A}=FP_{R} $ is an element of $\tau (FP_{R}).$ Now, let $O_{\delta
_{1}},O_{\delta _{2}}\in FP_{R}$ with $\delta _{1}$ and $\delta _{2}$ in $%
FI_{R}.$ Then $O_{\delta _{1}}\cap O_{\delta _{2}}=\{\mu \in FP_{R}:\delta
_{1}\nleq \mu $ and $\delta _{2}\nleq \mu \}.$ Since $R$ is fully
idempotent, therefore, $\delta _{1}\delta _{2}=\delta _{1}\wedge \delta
_{2}, $ by \ref{THR 2.1}. Since $\mu $ is fuzzy prime, so $\delta _{1}\delta
_{2}\leq \mu $ implies that $\delta _{1}\leq \mu $ or $\delta _{2}\leq \mu .$
Hence, it follows that $\delta _{1}\delta _{2}\nleq \mu ,$ that is, $\delta
_{1}\wedge \delta _{2}\nleq \mu .$ Conversely, $\delta _{1}\wedge \delta
_{2}\nleq \mu ,$ obviously, implies that $\delta _{1}\nleq \mu $ and $\delta
_{2}\nleq \mu .$ Thus the statements $\delta _{1}\nleq \mu $ and $\delta
_{2}\nleq \mu ,$ and $\delta _{1}\wedge \delta _{2}\nleq \mu ,$ are
equivalent. Hence

$O_{\delta _{1}}\cap O_{\delta _{2}}=\{\mu \in FP_{R}:\delta _{1}\wedge
\delta _{2}\nleq \mu \}=O_{\delta _{1}\wedge \delta _{2}}.$

Let us now consider an arbitrary family

$\{\eta _{i}\}_{i\in I}$ of fuzzy hyperideals of $R.$ Since

$\underset{i\in I}{\cup }O_{\eta _{i}}=\underset{i\in I}{\cup }\{\mu \in
FP_{R}:\eta _{i}\nleq \mu \}=\{\mu \in FP_{R}:\exists $ $k\in I$ so that $%
\eta _{i}\nleq \mu \}.$

Note that

$(\underset{i\in I}{\Sigma }\eta _{i})(x)=\underset{x\in x_{1}\oplus
x_{2}\oplus x_{3}\oplus ...}{\vee }(\eta _{1}(x_{1})\wedge \eta
_{2}(x_{2})\wedge \eta _{3}(x_{3})\wedge ...)$

where only a finite number of the $x_{i}^{\prime }s$ are not $0.$ Thus,
since $\eta _{i}(0)=1,$ therefore, we are considering the infimum of a
finite number of terms, because the $1^{\prime }s$ are effectively not being
considered.

Now, if for some $k\in I,$ $\eta _{k}\nleq \mu ,$ then there exists $x\in R$
such that $\eta _{k}>\mu .$ Consider the particular factorization of $x$ for
which $x_{k}=x$ and $x_{i}=0$ for all $i\neq k.$ We see that $\eta _{k}(x)$
is an element of the set whose supremum is defined to be $(\underset{i\in I}{%
\Sigma }\eta _{i})(x).$ Thus, $(\underset{i\in I}{\Sigma }\eta _{i})(x)\geq
\eta _{k}(x)>\mu (x).$ Thus $(\underset{i\in I}{\Sigma }\eta _{i})(x)>\mu
(x).$ Hence, we have $\underset{i\in I}{\Sigma }\eta _{i}\nleq \mu .$

Hence, $\eta _{k}\nleq \mu $ for some $k\in I$ implies that $\underset{i\in I%
}{\Sigma }\eta _{i}\nleq \mu .$

Conversely, suppose that $\underset{i\in I}{\Sigma }\eta _{i}\nleq \mu .$
Therefore, there exists an element $x\in R$ such that $(O)(x)>\mu (x).$ This
means that $\underset{x\in x_{1}\oplus x_{2}\oplus x_{3}\oplus ...}{\vee }%
(\eta _{1}(x_{1})\wedge \eta _{2}(x_{2})\wedge \eta _{3}(x_{3})\wedge
...)>\mu (x).$

Now, if all the elements of the set, whose supremum we are taking, are
individually less than or equal to $\mu (x),$ then we have $(\underset{i\in I%
}{\Sigma }\eta _{i})(x)=\underset{x\in x_{1}\oplus x_{2}\oplus x_{3}\oplus
...}{\vee }(\eta _{1}(x_{1})\wedge \eta _{2}(x_{2})\wedge \eta
_{3}(x_{3})\wedge ...)\leq \mu (x)$

which does not agree with what we have assumed. Thus, there is at least one
element of the set (whose supremum we are taking), say, $\eta
_{1}(x_{1}^{\prime })\wedge \eta _{2}(x_{2}^{\prime })\wedge \eta
_{3}(x_{3}^{\prime })\wedge ...$which is greater than $\mu (x)(x\in
x_{1}^{\prime }\oplus x_{2}^{\prime }\oplus x_{3}^{\prime }\oplus ...$being
the corresponding breakup of $x,$ where only a finite number of the $%
x_{s}^{\prime }$'$s$ are nonzero$).$ Thus $\eta _{1}(x_{1}^{\prime })\wedge
\eta _{2}(x_{2}^{\prime })\wedge \eta _{3}(x_{3}^{\prime })\wedge ...>\mu
(x)\geq \mu (x_{1}^{\prime })\wedge \mu (x_{2}^{\prime })\wedge \mu
(x_{3}^{\prime })\wedge ...$

That is, $\eta _{1}(x_{1}^{\prime })\wedge \eta _{2}(x_{2}^{\prime })\wedge
\eta _{3}(x_{3}^{\prime })\wedge ...>\mu (x_{1}^{\prime })\wedge \mu
(x_{2}^{\prime })\wedge \mu (x_{3}^{\prime })\wedge ...$

That is, $\eta _{1}(x_{1}^{\prime })\wedge \eta _{2}(x_{2}^{\prime })\wedge
\eta _{3}(x_{3}^{\prime })\wedge ...>\mu (x_{p}^{\prime })$

where $\mu (x_{p}^{\prime })=\mu (x_{1}^{\prime })\wedge \mu (x_{2}^{\prime
})\wedge \mu (x_{3}^{\prime })\wedge ...(p\in I).$

Hence $\eta _{1}(x_{1}^{\prime })>\mu (x_{p}^{\prime }).$ It follows that $%
\eta _{p}\nleq \mu $ for some $p\in N.$ Hence, $\underset{i\in I}{\Sigma }%
\eta _{i}\nleq \mu \Rightarrow \eta _{p}\nleq \mu $ for some $p\in N.$
Hence, the two statements, that is,

(i) $\underset{i\in I}{\Sigma }\eta _{i}\nleq \mu ,$ and

(ii) $\eta _{p}\nleq \mu $ for some $p\in I$ are equivalent. Hence

$\underset{i\in I}{\cup }O_{\eta _{i}}=\underset{i\in I}{\vee }\{\mu \in
FP_{R}:\eta _{i}\nleq \mu \}$

$=\underset{i\in I}{\cup }\{\mu \in FP_{R}:\underset{i\in I}{\Sigma }\eta
_{i}\nleq \mu \}$

$=O_{\underset{i\in I}{\Sigma }\eta _{i}}$

because, $\underset{i\in I}{\Sigma }\eta _{i}$ is also a fuzzy hyperideal of 
$R.$ Thus, $\underset{i\in I}{\cup }O_{\eta _{i}}\in \tau (FP_{R}).$ Hence
it follows that $\tau (FP_{R})$ forms a topology on the set $FP_{R}.$ Let $%
\phi :FI_{R}\longrightarrow FP_{R}$ be the mapping defined by $\lambda
\longrightarrow O_{\lambda }.$ It follows from the above that the
prescription $\phi (\lambda )=O_{\lambda }$ preserves finite intersection
and arbitrary union. Thus $\phi $ is a lattice homomorphism. To conclude the
proof, we must show that $\phi $ is bijective. In fact, we need to prove the
equivalence $\delta _{1}=\delta _{2},$ if and only if $O_{\delta
_{1}}=O_{\delta _{2}},$ for $\delta _{1},$ $\delta _{2}$ in $L_{A}.$ Suppose
that $O_{\delta _{1}}=O_{\delta _{2}}.$ If $\delta _{1}\neq \delta _{2},$
then there exists $x\in R$ such that $\delta _{1}(x)\neq \delta _{2}(x).$
Thus, either $\delta _{1}(x)>\delta _{2}(x)$ or $\delta _{2}(x)>\delta
_{1}(x).$ Suppose that $\delta _{1}(x)>\delta _{2}(x).$ Using Lemma \ref%
{LEMMA1}, there exists a fuzzy prime hyperideal $\mu $ of $R$ such that $%
\delta _{2}\leq \mu $ and $\delta _{2}(x)=\mu (x).$ Hence, $\delta _{1}\nleq
\mu ,$ because $\delta _{1}(x)>\delta _{2}(x)=\mu (x).$ Therefore, $\delta
_{1}(x)>\mu (x).$ Thus, $\mu \in O_{\delta _{1}}$. Our assumption is that $%
O_{\delta _{1}}=O_{\delta _{2}}.$ Hence, we have $\mu \in O_{\delta _{2}}.$
Hence $\delta _{2}\nleq \mu .$ this is a contradiction. If, in the
beginning, we had assumed that $\delta _{2}(x)>\delta _{1}(x)$ then, again,
by similar reasoning, we get a contradiction. Thus, $O_{\delta
_{1}}=O_{\delta _{2}}$ implies that $\delta _{1}=\delta _{2}.$ Conversely,
if $\delta _{1}=\delta _{2}$, then, by definition, it is obvious that $%
O_{\delta _{1}}=O_{\delta _{2}}.$ Thus, we have proved that $\delta
_{1}=\delta _{2}$ if and only if $O_{\delta _{1}}=O_{\delta _{2}}$ for $%
\delta _{1}$and $\delta _{2}$ in $L_{A}.$ This completes the proof of the
theorem.

The set $FP_{R}$ will be called the fuzzy prime spectrum of $R$ and the
topology $\tau (FP_{R})$ constructed in the above theorem will be called the
fuzzy spectral topology on $FP_{R}.$ The associated topological space will
be called the fuzzy spectral space of $R.$
\end{proof}

\end{document}